\documentclass[12pt]{amsart}
\usepackage{graphicx}
\usepackage{amsmath, amsthm, amssymb,amsfonts, mathrsfs}
\usepackage[cal=euler]{mathalpha}
\usepackage{mathtools}
\usepackage{tikz}
\usetikzlibrary{decorations.pathreplacing,calligraphy,calc}
\usepackage{tikz-cd}
\usepackage{color} 
\definecolor{rossred}{rgb}{1.0,0.25,0.66}  
\definecolor{rossgreen}{rgb}{0.25,0.66,0.25} 
\definecolor{rossblue}{rgb}{0.25,0.66,1.0}
\definecolor{sashapurple}{rgb}{0.5,0.15,0.5}
\usepackage[pdfborder={0 0 0}]{hyperref}
\hypersetup{
    colorlinks=true,
    linkcolor=rossblue,
    citecolor=rossgreen
}
\usepackage{url}
\usepackage{enumitem}
\usepackage[utf8]{inputenc}
\usepackage[T1]{fontenc}
\usepackage[paperwidth = 8.5in, paperheight = 11in, inner = 1in, outer = 1in, top = 1in, bottom = 1in]{geometry}

\numberwithin{equation}{section}
\newcommand{\KK}{{\mathbb{K}}}

\newcommand{\ZZ}{{\mathbb{Z}}}

  \newcommand{\C}{{\mathcal{C}}}

  \newcommand{\N}{{\mathcal{N}}}

\theoremstyle{plain}
\newtheorem{theorem}{Theorem}[section]

\newtheorem{lemma}[theorem]{Lemma}

\newtheorem{corollary}[theorem]{Corollary}

\theoremstyle{definition}
\newtheorem{defn}[theorem]{Definition}

\newtheorem{remark}[theorem]{Remark}

\newtheorem{example}[theorem]{Example}
\AtBeginEnvironment{example}{%
  \pushQED{\qed}%
}
\AtEndEnvironment{example}{\popQED\endexample}

\title{The $h$-vectors of toric ideals of odd cycle compositions revisited}
\author{Kieran Bhaskara}
\address[K. Bhaskara]
{Department of Mathematics and Statistics
McMaster University, Hamilton, ON L8S 4L8, Canada}
\email{kieran.bhaskara@mcmaster.ca}

\author{Adam Van Tuyl}
\address[A. Van Tuyl]
{Department of Mathematics and Statistics
McMaster University, Hamilton, ON L8S 4L8, Canada}
\email{vantuyla@mcmaster.ca}

\author{Sasha Zotine}
\address[S. Zotine]
{Department of Mathematics and Statistics
McMaster University, Hamilton, ON L8S 4L8, Canada}
\email{zotinea@mcmaster.ca}

\date{\today}

\begin{document}
\begin{abstract}
    Let $G$ be a graph consisting of $s$ odd cycles that all share a common
    vertex.  Bhaskara, Higashitani, and Shibu Deepthi recently computed
    the $h$-polynomial for the quotient ring $R/I_G$, where $I_G$ is
    the toric ideal of $G$, in terms of the number and sizes of odd cycles in the graph.
    The purpose of this note is to prove the stronger result that these toric ideals are 
    geometrically vertex decomposable, which allows us to deduce
    the result of Bhaskara, Higashitani, and Shibu Deepthi about the
    $h$-polyhomial as a corollary.
    \end{abstract}

\maketitle

\section{Introduction}
\label{sec:intro}

Toric ideals of graphs provide a fruitful bridge between combinatorics
and commutative algebra.   Recall that if
$G = (V,E)$ is a finite simple graph
on the vertex set $V = \{x_1,\ldots,x_n\}$ and edge set $E = \{e_1,\ldots,e_m\}$,
then the {\it toric ideal} of $G$ is the kernel of
the $\mathbb{K}$-algebra homomorphism 
$\varphi: \mathbb{K}[e_1,\ldots,e_m] \rightarrow \mathbb{K}[x_1,\ldots,x_n]$
given by $\varphi(e_i) = x_jx_k$ if $e_i = \{x_j,x_k\}$.  We denote
${\rm ker}(\varphi)$ by $I_G$ and write $\mathbb{K}[G]$ for
$\mathbb{K}[e_1,\ldots,e_m]/I_G$, which is sometimes
referred to as the {\it toric edge ring} of $G$.  Studying how the 
combinatorics of $G$ are reflected in $\mathbb{K}[G]$, 
and vice-versa, is an active area of research; see \cite{ADS, BOKVT2017,
CDSRVT2023, HHBOK2019,HHKOK2011,V} for a small
sample of this work.

A recent example of this interaction is work of  
Bhaskara-Higashitani-Shibu Deepthi \cite{BHSD2023edgerings}, generalizing 
earlier work of
Higashitani-Shibu Deepthi \cite{HD2023}, on the $h$-polynomial of
the toric edge rings of the family of graphs consisting of odd cycles 
that share a common vertex.   More precisely, fix an
integer $s \geq 1$ and let $t_1,\ldots,t_s$ be positive integers.  
Let $G = G(t_1,\ldots,t_s)$ be the graph consisting of $s$ odd cycles of
size $2t_i+1$ for $i=1,\ldots,s$, where all the odd cycles share a common
``center'' vertex.   An example of the graph $G(2,1,2)$ is given
in Figure \ref{fig.g(2,1,2)}.   
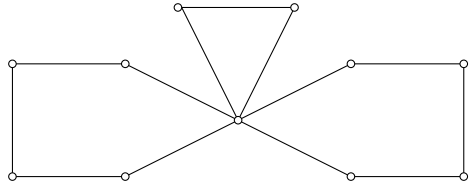
\begin{figure}[htp!]
\begin{tikzpicture}[scale=0.5]
\draw (6,0) -- (6,3) -- (9,3) -- (12,1.5)  -- (9,0) -- (6,0);
\draw (12,1.5) -- (15,3) -- (18,3) -- (18,0) -- (15,0) -- (12,1.5);
\draw (12,1.5) -- (10.5,4.5) -- (13.5,4.5) -- (12,1.5);

\fill[fill=white,draw=black] (6,0) circle (.1)  {};
\fill[fill=white,draw=black] (9,0) circle (.1)  {};
\fill[fill=white,draw=black] (6,3) circle (.1)  {};
\fill[fill=white,draw=black] (9,3) circle (.1)  {};
\fill[fill=white,draw=black] (15,0) circle (.1) {};
\fill[fill=white,draw=black] (15,3) circle (.1)  {};
\fill[fill=white,draw=black] (18,0) circle (.1) {};
\fill[fill=white,draw=black] (18,3) circle (.1) {};
\fill[fill=white,draw=black] (10.4,4.5) circle (.1)  {};
\fill[fill=white,draw=black] (12,1.5) circle (.1) {};
\fill[fill=white,draw=black] (13.5,4.5) circle (.1) {};
\end{tikzpicture}
    \caption{The graph $G(2,1,2)$ consisting of two $2\cdot 2+1 = 5$-cycles
    and one $2\cdot 1 + 1 =3$-cycle meeting at a common vertex}
    \label{fig.g(2,1,2)}
\end{figure}
The main
result of \cite{BHSD2023edgerings} is to compute
the $h$-polynomial of $\mathbb{K}[G(t_1,\ldots,t_s)]$,
that is, the numerator of the reduced Hilbert series 
of $\mathbb{K}[G(t_1,\ldots,t_s)]$ (see the next
section or \cite{BHSD2023edgerings} for any undefined terms).
Note that when $t_1=\cdots=t_s =1$, the graph $G(1,\ldots,1)$ consists of $s$ triangles that share a common vertex---this  family is the one studied in \cite{HD2023}.

The proof strategy in both \cite{BHSD2023edgerings} and
\cite{HD2023} is to consider the initial
ideal of the toric ideal of $G(t_1,\ldots,t_s)$. 
For this family of graphs, the initial ideal is a squarefree
monomial ideal and thus corresponds
to a simplicial complex via the Stanley--Reisner
correspondence. Through a careful study of 
the facets of this simplicial complex, the authors of \cite{BHSD2023edgerings,HD2023} construct a recursive 
formula for the Hilbert series, allowing them to deduce
their main result. 

The purpose of this article is to give an entirely
different proof that avoids passing to the initial
ideal of $I_G$.  In fact, we prove that
the ideals $I_G$ satisfy a stronger algebraic
property:

\begin{theorem}\label{maintheorem}
    The toric ideal $I_G$ of $G = G(t_1,\ldots,t_s)$ is geometrically vertex decomposable.
\end{theorem}

\noindent
Geometrically vertex decomposable ideals were
introduced by Klein and Rajchgot \cite{KR2021},
extending  work of Knutson--Miller--Yong \cite{KMY2009} on vertex decomposability.  
Very roughly speaking, geometrically vertex 
decomposable ideals generalize the properties of
Stanley--Reisner ideals of vertex decomposable
simplicial complexes to all ideals.  As shown
in \cite{CDSRVT2023, DH2023,  NRVT2024gvdideals}, geometrically vertex decomposable
ideals are a useful addition to
the commutative algebra toolkit.

Theorem~\ref{maintheorem} also demonstrates the
advantage of identifying geometrically vertex 
decomposable ideals.  Indeed, using the properties of the Hilbert series of quotients by geometrically vertex decomposable ideals developed by 
Nguy{\~{\^e}}n--Rajchgot--Van Tuyl \cite{NRVT2024gvdideals}, the main results
of \cite{BHSD2023edgerings,HD2023} are now
corollaries of Theorem~\ref{maintheorem}.

\begin{corollary}[{\cite[Theorem~1.1]{BHSD2023edgerings}}]\label{maincor}
    The h-polynomial of $\KK[G(t_1,\ldots,t_s)]$ is
    \begin{equation}
        h \bigl( \KK[G(t_1,\ldots,t_s)] ; z\bigr) = \prod_{i=1}^s (1 + z + z^2 + \cdots + z^{t_i}) - z \prod_{i=1}^s (1 + z + z^2 + \cdots + z^{t_i - 1}).\label{main formula}
    \end{equation}
    
\end{corollary}
\noindent
As a second application, we 
compute the Castelnuovo-Mumford regularity
of $\mathbb{K}[G(t_1,\ldots,t_s)]$. 

\begin{corollary}
\label{secondcor}
The Castelnuovo-Mumford regularity of 
$\KK[G(t_1,\ldots,t_s)]$ is
$${\rm reg}(\KK[G(t_1,\ldots,t_s)]) =
\begin{cases}
t_1+\cdots +t_s & ~~\mbox{if $s \geq 2$} \\
0 & ~~\mbox{if $s=1$.}
\end{cases}
$$
\end{corollary}

In the next section we recall the relevant background and facts on geometrically vertex decomposable ideals and 
$h$-polynomials. In Section~\ref{sec:mainresults}, we prove Theorem~\ref{maintheorem} and its consequences, Corollary~\ref{maincor} and Corollary~\ref{secondcor}.


\section{Geometrically vertex decomposable ideals and 
Hilbert series}
\label{sec:background}

Throughout this paper $\mathbb{K}$ will
denote a field. Let $R = \mathbb{K}[x_1,\ldots,x_n]$ be a polynomial
ring and fix a variable $y \in R$.  For any $f \in R$,
there is some integer $d \geq 0$
such that we can write $f$ as $f = \sum_{i=0}^d \alpha_iy^i$ where
$\alpha_i$ is a polynomial in the variables
$\{x_1,\ldots,x_n\} \setminus \{y\}$. The {\it initial
$y$-form} of $f$ is ${\rm in}_y(f) = \alpha_dy^d$.  
For any ideal $I$ of $R$,
the ideal ${\rm in}_y(I)  = \langle {\rm in}_y(f) ~|~
f \in I \rangle$ is the ideal generated by
all the initial $y$-forms of $f$ in the ideal $I$.
A monomial order $<$ is {\it $y$-compatible} if
${\rm in}_<(f) = {\rm in}_<({\rm in}_y(f))$ for all
$f \in R$.

Fixing a $y$-compatible monomial order $<$, suppose
that $\mathcal{G}(I) = \{g_1,\ldots,g_t\}$ is 
a Gr\"obner basis of $I$ with resect to $<$.
For each $g_i$ in the Gr\"obner basis, we write it
as $g_i = q_iy^{d_i} + r_i$ where ${\rm in}_y(g_i) = 
q_iy^{d_i}$.  Note that this means $y$ does not divide
any term of $q_i$ and $y^{d_i}$ does not divide
any term of $r_i$.   With this notation, we define
two new ideals:
\[\mathcal{C}_{y,I} = \langle q_1,\ldots,q_t \rangle
~~\mbox{and}~~ \mathcal{N}_{y,I} = \langle q_i ~|~ d_i = 0 \rangle. \]
We now recall the definition of
a geometrically vertex decomposable ideal.
\begin{defn}[{\cite[Definition 2.7]{KR2021}}]
\label{defn.gvd}
    An ideal $I$ of $R = \mathbb{K}[x_1,\ldots,x_n]$ is
    {\it geometrically vertex decomposable} (GVD) if 
    $I$ is unmixed (that is, all of the associated
    primes of $I$ have the same height) and 
    \begin{enumerate}
        \item $I = \langle 1 \rangle$, or $I$ is
        generated by a (possible empty) subset of
        variables, or
        \item there exists a variable $y$
        and $y$-compatible monomial order such
        that 
        $$ {\rm in}_y(I) = \mathcal{C}_{y,I} \cap (\mathcal{N}_{y,I} +
        \langle y \rangle)$$
        and the contractions of $\mathcal{C}_{y,I}$ and $\mathcal{N}_{y,I}$
        to $\mathbb{K}[x_1,\ldots,\hat{y},\ldots,x_n]$ 
        are geometrically vertex decomposable.
    \end{enumerate}
\end{defn}

The following lemma allows
us to conclude the equality ${\rm in}_y(I) = \mathcal{C}_{y,I} \cap \bigl( \mathcal{N}_{y,I} + \langle y \rangle \bigr)$ automatically holds as long as we preserve a suitably nice generating set.
\begin{lemma}\cite[Theorem 2.1]{KMY2009}\label{lem:sqfreelemma}
Let $I$ be an ideal of $R$, set $<$ to be a $y$-compatible monomial order, and assume $\{g_1, \ldots, g_t \}$ is a Gr\"obner basis of $I$ with respect to $<$. Suppose that ${\rm in}_y (g_i) = q_iy^{d_i}$ with $d_i = 0$ or $1$ for all i. Then
\begin{enumerate}
\item  $\{q_1, \ldots, q_t\}$ is a Gr\"obner basis of 
$\mathcal{C}_{y,I}$, $\{q_i ~|~ d_i = 0\}$ is a Gr\"obner basis of 
$\mathcal{N}_{y,I}$, and
\item  ${\rm in}_y(I) = \mathcal{C}_{y,I} \cap \bigl( \mathcal{N}_{y,I} + \langle y \rangle \bigr).$
\end{enumerate}
\end{lemma}

For our proof of Theorem~\ref{maintheorem}, we need that certain complete intersections are GVD.

\begin{lemma}[{\cite[Corollary 2.12]{CDSRVT2023}}]
\label{lem.CIGVD}
If $I$ is a squarefree monomial ideal that 
is also a complete intersection, then $I$ 
is geometrically vertex decomposable.
\end{lemma}

We now turn our attention to the Hilbert series and 
the $h$-polynomial. Equip $R = \KK[x_1,\ldots,x_n]$ with the standard grading and let $I$ be a homogeneous ideal of $R$. The \textit{Hilbert series} of $R/I$ is the formal power series $\textup{HS}(R/I;z) = \sum_{i\geq0} \dim_\mathbb{K}(R/I)_{i} z^i$,
where $\dim_\mathbb{K}(R/I)_i$ denotes the dimension of the $i$-th graded piece of $R/I$. The Hilbert-Serre Theorem (\cite[Theorem~5.1.4]{V} and \cite[Lemma 4.1.13]{BH}) guarantees that the Hilbert series can be written as a rational function
\[
\textup{HS}(R/I;z)=\frac{p(z)}{(1-z)^n}
\]
for some $p(z) \in \ZZ[z]$. Moreover, we may simplify this rational function and obtain the \textit{reduced Hilbert series}
 \[ \textup{HS}(R/I;z)=\frac{h(R/I;z)}{(1-z)^{\dim(R/I)}},
    \] where $h(R/I;z)$ is the unique polynomial in $\ZZ[z]$ satisfying $h(R/I;1) \neq 0$. This polynomial $h(R/I;z)=h_0+h_1z+\cdots +h_mz^m$ with $h_m \neq 0$ is called the {\it $h$-polynomial} of $R/I$.

The following fact about complete intersections
is well-known. 
\begin{lemma}
\label{lem:h-poly-CI}
   Suppose $I = \langle f_1,\ldots,f_r \rangle 
   \subseteq R$  
   is a complete intersection with $\deg f_i =d_i$ for each $i=1,\dots,r$.  Then 
   $$h(R/I;z) =\prod_{i=1}^r (1+z+ \cdots + z^{d_i-1}).
   $$ 
\end{lemma}

The next result links $h$-polynomials and geometrically vertex decomposable ideals.

\begin{theorem}[{\cite[Theorem~2.4]{NRVT2024gvdideals}}]
\label{thm:hpolygvd}
Let $I \subseteq R$ be a homogeneous GVD ideal and suppose that $y$ is
the variable such that there is a $y$-compatible
monomial order such that 
$\operatorname{in}_y(I) = \C_{y,I} \cap \bigl( \N_{y,I} + \langle y \rangle \bigr)$.
If $\C_{y,I} \neq \langle 1 \rangle$ and $\sqrt{\C_{y,I}} \neq \sqrt{\N_{y,I}}$, then
the $h$-polynomial of $R/I$ satisfies
\begin{equation*}
    h(R/I;z) = h(R/\N_{y,I};z) + z\cdot h(R/\C_{y,I};z).
\end{equation*}
\end{theorem}

\section{Main results}
\label{sec:mainresults}

In this section, we prove Theorem~\ref{maintheorem}. To that end, we introduce notation. Our graph $G = G(t_1, \ldots, t_s)$ consists of $s$ odd cycles $C_{2t_i+1}$ of size $2t_i + 1$ for $i = 1, \ldots s$. We use $e_{i,j}$ to denote the $j$-th edge of the $i$-th cycle $C_{2t_i+1}$ read clockwise starting from the central vertex; see Figure~\ref{fig.g(ki)}.

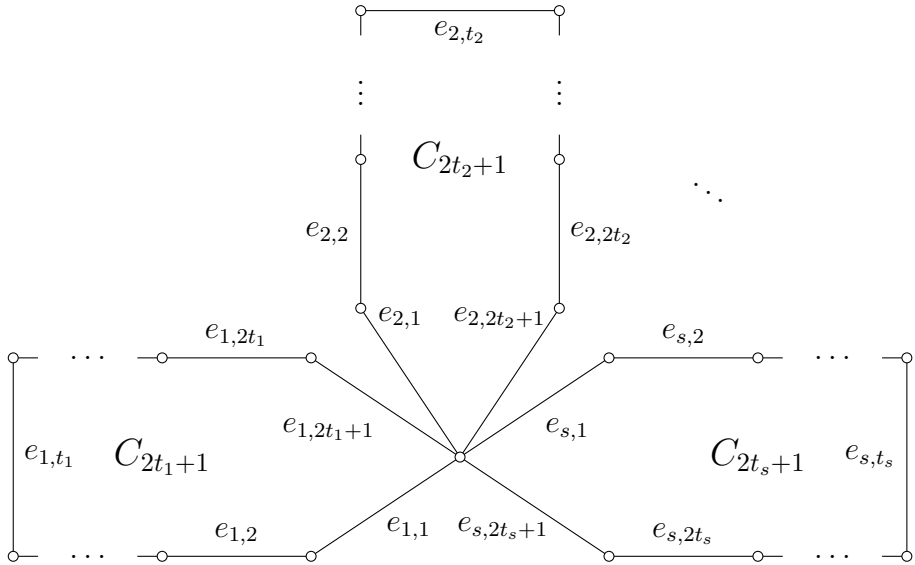
\begin{figure}[htp!]
\begin{tikzpicture}[scale=0.66]
\draw (9,3) -- (6,1) -- (2.5,1);
\node at (1.5,1) {$\ldots$};
\draw (0.5, 1) -- (0,1) -- (0,5) -- (0.5,5);
\node at (1.5,5) {$\ldots$};
\draw (2.5,5) -- (6,5) -- (9,3);
\draw (9,3) -- (7,6) -- (7,9.5);
\node at (7,10.5) {$\vdots$};
\draw (7, 11.5) -- (7,12) -- (11,12) -- (11,11.5);
\node at (11,10.5) {$\vdots$};
\draw (11,9.5) -- (11,6) -- (9,3);
\draw (9,3) -- (12,1) -- (15.5,1);
\node at (16.5,1) {$\ldots$};
\draw (17.5, 1) -- (18,1) -- (18,5) -- (17.5,5);
\node at (16.5,5) {$\ldots$};
\draw (15.5,5) -- (12,5) -- (9,3);

\fill[fill=white,draw=black] (9,3) circle (.1) {};
\fill[fill=white,draw=black] (6,1) circle (.1) {};
\fill[fill=white,draw=black] (3,1) circle (.1) {};
\fill[fill=white,draw=black] (0,1) circle (.1) {};
\fill[fill=white,draw=black] (0,5) circle (.1) {};
\fill[fill=white,draw=black] (3,5) circle (.1) {};
\fill[fill=white,draw=black] (6,5) circle (.1) {};
\fill[fill=white,draw=black] (7,6) circle (.1) {};
\fill[fill=white,draw=black] (7,9) circle (.1) {};
\fill[fill=white,draw=black] (7,12) circle (.1) {};
\fill[fill=white,draw=black] (11,6) circle (.1) {};
\fill[fill=white,draw=black] (11,9) circle (.1) {};
\fill[fill=white,draw=black] (11,12) circle (.1) {};
\fill[fill=white,draw=black] (18,1) circle (.1) {};
\fill[fill=white,draw=black] (15,1) circle (.1) {};
\fill[fill=white,draw=black] (12,1) circle (.1) {};
\fill[fill=white,draw=black] (12,5) circle (.1) {};
\fill[fill=white,draw=black] (15,5) circle (.1) {};
\fill[fill=white,draw=black] (18,5) circle (.1) {};

\node at (3,3) {{\large $C_{2t_1+1}$}};
\node at (9,9) {{\large $C_{2t_2+1}$}};
\node at (14,8.5) {$\ddots$};
\node at (15,3) {{\large $C_{2t_s+1}$}};
\node[below right] at (7.3,2) {$e_{1,1}$};
\node[above] at (4.5,1) {$e_{1,2}$};
\node[right] at (0,3) {$e_{1,t_1}$};
\node[above] at (4.5,5) {$e_{1,2t_1}$};
\node[below left] at (7.5,4) {$e_{1,2t_1+1}$};
\node at (7.8,5.8) {$e_{2,1}$};
\node at (9.8,5.8) {$e_{2,2t_2+1}$};
\node[left] at (7,7.5) {$e_{2,2}$};
\node[below] at (9,12) {$e_{2,t_2}$};
\node[right] at (11,7.5) {$e_{2,2t_2}$};
\node[below left] at (11,2) {$e_{s,2t_s+1}$};
\node[above] at (13.5,1) {$e_{s,2t_s}$};
\node[left] at (18,3) {$e_{s,t_s}$};
\node[above] at (13.5,5) {$e_{s,2}$};
\node[below right] at (10.5,4) {$e_{s,1}$};
\end{tikzpicture}
    \caption{The graph $G = G(t_1,t_2,\ldots,t_s)$ and the notation for its edges.}
    \label{fig.g(ki)}
\end{figure}

\begin{remark}
    Note that our notation for the graph $G(t_1,\ldots,t_s)$
    is different than that found in \cite{BHSD2023edgerings}.
    In that paper the authors write $G(r_1,\ldots,r_t)$
    where $r_i$ is the number of odd-cycles of length
    $2i+1$.  We also want to highlight that we make
    no assumptions on the order of $t_1,\ldots,t_s$.
    However, it should be clear that for any
    permutation $\sigma \in S_s$, the graphs
    $G(t_1,\ldots,t_s)$ and $G(t_{\sigma(1)},\ldots, t_{\sigma(s)})$ are isomorphic graphs.
\end{remark}

With this notation, we can express the universal Gr\"obner basis of the toric ideal $I_G$.  

\begin{lemma}
\label{lemma:univGB}
If $G = G(t_1,\ldots,t_s)$, then the
set 
\begin{equation*}
\left\{ \Bigl( \prod_{i=0}^{t_p} e_{p,2i+1} \Bigr) \Bigl( \prod_{j=1}^{t_q} e_{q,2j} \Bigr) - \Bigl( \prod_{k=0}^{t_q} e_{q,2k+1} \Bigr) \Bigl( \prod_{\ell=1}^{t_p} e_{p,2\ell} \Bigr) \middle\vert 1 \leq p < q \leq s \right\}
\end{equation*}
is a universal Gr\"obner basis 
of $I_G$.
\end{lemma}

\begin{proof}[Sketch of the proof] By {\cite[Proposition 10.1.10]{V}}
the set of primitive closed even walks in a graph $G$
are used to define the binomials that form a universal Gr\"obner basis of $I_G$.  As shown in \cite{BHSD2023edgerings}, all the primitive closed
even walks of $G(t_1,\ldots,t_s)$ have the form
\[(e_{p,i},e_{p,2},\ldots,e_{p,2t_p+1},e_{q,1},\ldots, 
e_{q,2t_q+1})\] for some $p < q$.  
Informally, the primitive 
closed even walk involves starting at the center vertex,
traveling around the odd cycle $C_{2t_p+1}$, returning
to the center, and then traveling around the 
cycle $C_{2t_q+1}$, and  back to the center.   This primitive
closed even walk then corresponds to the
binomial in the universal Gr\"obner basis.
\end{proof}

Before presenting our proof of Theorem~\ref{maincor},
we provide  an example to demonstrate the structure of the argument. Specifically, we show the toric ideal of the graph $G(2,1,2)$ is geometrically vertex decomposable.

\begin{example}
\label{ex:g(2,1,2)}
The graph $G(2,1,2)$ is pictured in Figure~\ref{fig.g(2,1,2)}. Using Lemma~\ref{lemma:univGB}, the universal Gr\"obner basis for $I\coloneq I_{G(2,1,2)}$ is
\begin{align*}
    I_{G(2,1,2)} = \bigl \langle &e_{1,1} e_{1,3} e_{1,5} e_{2,2} - e_{1,2} e_{1,4} e_{2,1} e_{2,3}, \\ &e_{1,1} e_{1,3} e_{1,5} e_{3,2} e_{3,4} - e_{1,2} e_{1,4} e_{3,1} e_{3,3} e_{3,5}, \\ &e_{2,1} e_{2,3} e_{3,2} e_{3,4} - e_{2,2} e_{3,1} e_{3,3} e_{3,5} \bigr \rangle.
\end{align*}
To show that this ideal is GVD, we use an induction on the number of cycles of $G$. Fix a lexicographic monomial order where the variables corresponding to odd edges in the last cycle are of highest weight, i.e. $e_{3,5} > e_{3,3} > e_{3,1} > \cdots$. This order is chosen so that it is always $y$-compatible---it is precisely the order in which we select $y$. Indeed, we choose the term $y = e_{3,5}$ and compute the $\C$ and $\N$ ideals. To ease the notation, let $f_1$ denote the first generator of $I$. We explicitly compute that
\begin{align*}
    \C_0 \coloneqq \C_{e_{3,5}, I_{G(2,1,2)}} = \bigl \langle e_{2,2} e_{3,1} e_{3,3}, e_{1,2} e_{1,4} e_{3,1} e_{3,3}, f_1 \bigr \rangle && \N \coloneqq \N_{e_{3,5}, I_{G(2,1,2)}} = \bigl \langle f_1 \bigr \rangle.
\end{align*}
Each element $g_i$ in the Gr\"obner
basis of $I$  satisfies ${\rm in}_{e_{3,5}}(g_i) = e_{3,5}^{d_i}q_i$ with 
$d_i = 0$ or $1$.  Thus, by Lemma \ref{lem:sqfreelemma}, we have
${\rm in}_y(I) = \mathcal{C}_0 \cap \bigl( \mathcal{N} + \langle y \rangle \bigr)$,
and the listed generators of $\mathcal{C}_0$ and $\mathcal{N}$ are  Gr\"obner bases. Observe that $\N_{e_{3,5}, I}$ is just the toric ideal of the graph $G(2,1)$. By induction, we may assume that this ideal is GVD, so we just need to see that $\mathcal{C}_0$ is GVD. To do so, we now choose the term $y = e_{3,3}$ and compute again:
\begin{align*}
    \C_1 \coloneqq \C_{e_{3,3}, \C_0} = \bigl \langle e_{2,2} e_{3,1}, e_{1,2} e_{1,4} e_{3,1}, f_1 \bigr \rangle, && \N_{e_{3,3}, \C_0} = \N.
\end{align*}
We can apply Lemma \ref{lem:sqfreelemma} again to get
${\rm in}_y(\mathcal{C}_0) = \mathcal{C}_1 \cap \bigl( \mathcal{N} + \langle y \rangle \bigr)$,
and the listed generators of $\mathcal{C}_1$ form a Gr\"obner basis.
We automatically have that $\mathcal{N}$ is GVD by induction, so we just need $\mathcal{C}_1$ to be GVD. Finally, choosing the term $y = e_{3,1}$, we obtain
\begin{align*}
    \C_2 \coloneqq \C_{e_{3,1}, \C_1} = \bigl \langle e_{2,2}, e_{1,2} e_{1,4} \bigr \rangle, && \N_{e_{3,1}, \C_1} = \N.
\end{align*}
The required intersection again holds by Lemma \ref{lem:sqfreelemma}.
Observe that $f_1$ still lies in $\C_2$ but is now a redundant generator. Moreover, we see that $\C_2$ is a squarefree monomial complete intersection and hence is GVD by Lemma~\ref{lem.CIGVD}. This implies that all of the $\C_j$ are GVD and consequently the toric ideal $I_{G(2,1,2)}$ is GVD.

The example outlines the general framework of the argument, which we also illustrate with a diagram in Figure~\ref{fig.argument}. We induct on the number of cycles and the number of odd-indexed edges per cycle. For each odd-indexed edge, we split into the $\C$ and $\N$ ideals, the latter of which is always GVD by induction. The $\C$ ideals are GVD because the last step produces a squarefree monomial complete intersection, which is GVD by Lemma~\ref{lem.CIGVD}. We proceed with the precise formulation of this idea.
\end{example}

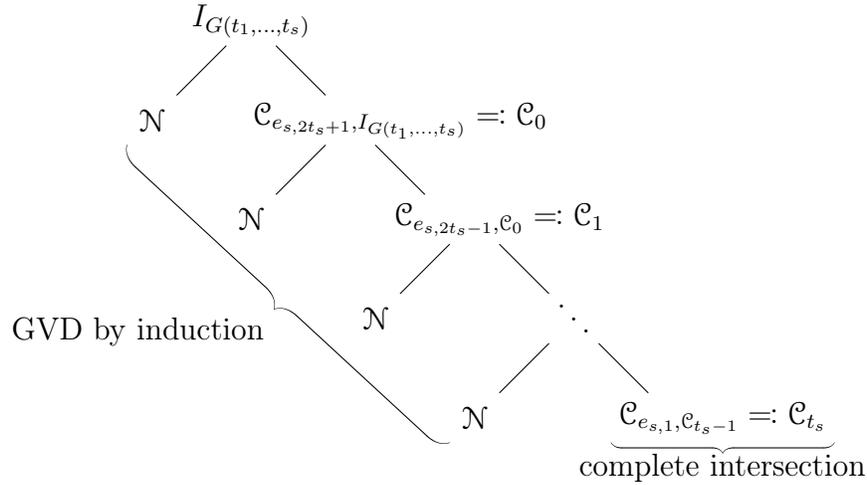
\begin{figure}[htp!]
\begin{tikzpicture}[scale=0.66]
\node at (0,0) {$I_{G(t_1,\ldots,t_s)}$};
\node at (3,-2) {$\C_{e_{s,2t_s+1}, I_{G(t_1,\ldots,t_s)}} \eqcolon \C_0$};
\node at (5,-4) {$\C_{e_{s,2t_s-1}, \C_0} \eqcolon \C_1$};
\node at (6.25,-5.75) {.};
\node at (6.5,-6) {.};
\node at (6.75,-6.25) {.};
\node at (9.5,-8) {$\C_{e_{s,1}, \C_{t_{s}-1}}\eqcolon \C_{t_s}$};
\draw (0.5,-0.5) -- (1.5, -1.5);
\draw (2.5,-2.5) -- (3.5, -3.5);
\draw (5,-4.5) -- (6, -5.5);
\draw (7,-6.5) -- (8, -7.5);
\node at (-2,-2) {$\N$};
\node at (0,-4) {$\N$};
\node at (2.5,-6) {$\N$};
\node at (4.5,-8) {$\N$};
\draw (-0.5,-0.5) -- (-1.5, -1.5);
\draw (1.5,-2.5) -- (0.5, -3.5);
\draw (4,-4.5) -- (3, -5.5);
\draw (6,-6.5) -- (5, -7.5);
\draw[decoration={calligraphic brace,mirror,amplitude=8pt},decorate]
  (-2.5,-2.5) -- node[below left=5pt] {GVD by induction} (4,-8.5);
  \draw[decoration={calligraphic brace,mirror,amplitude=4pt},decorate]
  (7.25,-8.5) -- node[below] {complete intersection} (11.75,-8.5);
\end{tikzpicture}
    \caption{Diagram summarizing the argument for why $I_{G(t_1,\ldots,t_s)}$ is GVD.}
    \label{fig.argument}
\end{figure}

\begin{proof}[Proof of Theorem~\ref{maintheorem}] 
We perform induction on the number of cycles $s$. The base case is $s = 1$, in which case the toric ideal is trivial and hence GVD.

Assume $s \geq 2$ and let $G = G(t_1, t_2, \ldots, t_s)$. Lemma~\ref{lemma:univGB} gives us that the toric ideal of $G$ is
\begin{equation}
\label{eqn:Igens}
    I \coloneqq I_{G(t_1,\ldots,t_s)} = \left\langle \Bigl( \prod_{i=0}^{t_p} e_{p,2i+1} \Bigr) \Bigl( \prod_{j=1}^{t_q} e_{q,2j} \Bigr) - \Bigl( \prod_{k=0}^{t_q} e_{q,2k+1} \Bigr) \Bigl( \prod_{\ell=1}^{t_p} e_{p,2\ell} \Bigr) \middle\vert 1 \leq p < q \leq s \right\rangle.
\end{equation}
Observe that $I_G$ is a prime ideal and hence unmixed. Fix the lexicographic monomial order with the variables ordered by
\begin{align*}
e_{s,2t_s+1} > e_{s,2t_s-1} > \cdots > e_{s,3} > e_{s, 1} > \text{remaining variables in any order}.
\end{align*}
By construction, this monomial order will always be $y$-compatible. Initialize the ideals
\begin{align}
\label{eqn:C-N-init}
    \C_0 \coloneqq \C_{e_{s,2t_s+1},I}, && \N_0 \coloneqq \N_{e_{s,2t_s+1},I},
\end{align}
and for each $1 \leq i \leq t_s$, construct the ideals
\begin{align}
\label{eqn:C-N-rec}
    \C_i \coloneqq \C_{e_{s,2(t_s-i)+1},\C_{i-1}}, && \N_i \coloneqq \N_{e_{s,2(t_s-i) + 1},\C_{i-1}}.
\end{align}
We explicitly compute the generators of these ideals, starting with $\C_0$ and $\N_0$. First, $\N_0$ is given by (c.f. Equation~\eqref{eqn:Igens})
\begin{equation*}
    \N_0 = \left\langle \Bigl( \prod_{i=0}^{t_p} e_{p,2i+1} \Bigr) \Bigl( \prod_{j=1}^{t_q} e_{q,2j} \Bigr) - \Bigl( \prod_{k=0}^{t_q} e_{q,2k+1} \Bigr) \Bigl( \prod_{\ell=1}^{t_p} e_{p,2\ell} \Bigr) \middle\vert 1 \leq p < q \leq s-1 \right\rangle = I_{G(t_1,t_2\ldots,t_{s-1})}
\end{equation*}
since any generator involving the last cycle $C_{2t_s+1}$ would have a term containing $e_{s,2t_s+1}$. It follows that
\begin{align*}
    \C_0 = \left\langle \Bigl( \prod_{k=0}^{t_s-1} e_{s,2k+1} \Bigr) \Bigl( \prod_{\ell=1}^{t_p} e_{p,2\ell} \Bigr) \middle\vert 1 \leq p \leq s-1 \right\rangle + \N_0.
\end{align*}
The same argument reveals that $\N_i = \N_0 = I_{G(t_1,t_2\ldots,t_{s-1})}$ for all $1 \leq i \leq t_s$. On the other hand, we have that
\begin{equation}
\label{eqn:Ci}
    \C_i = \left\langle \Bigl( \prod_{k=0}^{(t_s-i)-1} e_{s,2k+1} \Bigr) \Bigl( \prod_{\ell=1}^{t_p} e_{p,2\ell} \Bigr) \middle\vert 1 \leq p \leq s-1 \right\rangle + \N_0,
\end{equation}
since we are pulling out a single copy of the variable $e_{s,2(t_s - i)+1}$ at each step.  We verify that we satisfy the assumptions of Definition~\ref{defn.gvd} for being GVD. Observe that the generators in Equation~\eqref{eqn:Ci} satisfy ${\rm in}_{e_{s,2(t_s-i)+1}}(g) = (e_{s,2(t_s-i)+1})^dq$ with $d =0$ or 
$1$, and we have
${\rm in}_{e_{s,2t_s+1}}(I) = 
\mathcal{C}_{0} \cap \bigr( \mathcal{N}_{0} + \langle e_{s,2t_s+1}\rangle \bigl)$ by Lemma \ref{lem:sqfreelemma}. Proceeding by induction, we may assume that the generating set for $\mathcal{C}_{i-1}$ is a Gr\"obner basis. Therefore by Lemma \ref{lem:sqfreelemma}, we have that \eqref{eqn:Ci} forms a Gr\"obner basis for $\mathcal{C}_i$ and 
\[{\rm in}_{e_{s,2(t_s-i)+1}}(\mathcal{C}_{i-1}) = 
\mathcal{C}_{i} \cap \bigr( \mathcal{N}_{i} + \langle e_{s,2(t_s-i)+1}\rangle \bigl).\]

It remains to show that $\mathcal{N}_0$ and each $\mathcal{C}_i$ are GVD. The ideal $\mathcal{N}_0$, being the toric ideal of $G(t_1, t_2, \ldots, t_{s-1})$, is GVD by induction. On the other hand, when $i = t_s$, Equation~\ref{eqn:Ci} becomes
\begin{equation}
\label{eqn:Cts}
    \C_{t_s} = \left\langle \Bigl( \prod_{\ell=1}^{t_p} e_{p,2\ell} \Bigr) \middle\vert 1 \leq p \leq s-1 \right\rangle.
\end{equation}
Note that this ideal still contains $\N_0$, since both terms of every generator of $\N_0$ are divisible by a generator of
$\C_{t_s}$. Moreover, we see that $\C_{t_s}$ is a squarefree monomial ideal whose generators are all coprime, and hence form a regular sequence. 
Therefore $\C_{t_s}$ is a complete intersection and Lemma~\ref{lem.CIGVD} gives us that $\C_{t_s}$ is GVD. It follows that all of the $\C_i$ are GVD for $0 \leq i \leq t_s$ and hence so is $I$.
\end{proof}

We recover the results
of Bhaskara, Higashitani, and Shibu Deepthi from \cite{BHSD2023edgerings,HD2023}.

\begin{proof}[Proof of Corollary~\ref{maincor}]
We once again proceed by induction on the number of cycles $s$. For the base case $s=1$, the toric ideal is trivial, so $h \bigl( \KK[G(t_1)] ; z\bigr)=1$, which agrees with the right hand side of Equation~\eqref{main formula}.

Now let $s \geq 2$ and $G = G(t_1, t_2, \ldots, t_s)$. Let $\C_i$ and $\N_i$ be as defined in Equations~\eqref{eqn:C-N-init} and \eqref{eqn:C-N-rec}. As was observed in the proof, we have that all of the $\N_i$ are equal and so we denote $\N \coloneq \N_0$. Using Theorem~\ref{thm:hpolygvd}, we see that for each $0 \leq i \leq t_s - 1$, the $h$-polynomial of $\C_i$ is
\begin{equation*}
    h(R/\C_i ; z) = h(R/\N ; z) + z \cdot h(R/\C_{i+1} ; z).
\end{equation*}
It follows that the $h$-polynomial of $I_G$ is (c.f. Figure~\ref{fig.argument})
\begin{equation*}
    h(R/I_G ; z) = (1 + z + z^2 + \cdots + z^{t_s}) h(R/\N ; z) + z^{t_s+1} \cdot h(R/\C_{t_s} ; z).
\end{equation*}
Induction gives us that
\begin{equation*}
    h(R/\N ; z) = \prod_{i=1}^{s-1} (1 + z + \cdots + z^{t_i}) - z \prod_{i=1}^{s-1} (1 + z + \cdots + z^{t_i - 1}).
\end{equation*}
On the other hand, the ideal 
$\C_{t_s}$ of Equation~\eqref{eqn:Cts}
is a complete intersection, so 
Lemma~\ref{lem:h-poly-CI}  yields
\begin{equation*}
    h(R/\C_{t_s} ; z) = \prod_{i=1}^{s-1} (1 + z + \cdots + z^{t_i-1}).
\end{equation*}
Combining these, we see that
\begin{align*}
    h(R/I_G ; z) &= \left[ \prod_{i=1}^{s-1} (1 + z + \cdots + z^{t_i}) - z \prod_{i=1}^{s-1} (1 + z + \cdots + z^{t_i - 1}) \right] (1 + z + \cdots + z^{t_s}) \\ 
    &+ z^{t_s+1} \prod_{i=1}^{s-1} (1 + z + \cdots + z^{t_i - 1}). 
\end{align*}
Now observe that the middle term may be rewritten as
\begin{align*}
    \bigl( [1+z+\cdots+z^{t_s-1}] + z^{t_s} \bigr) z \prod_{i = 1}^{s-1} (1 + z + \cdots + z^{t_i - 1}) &= z \prod_{i=1}^s (1 + z + \cdots + z^{t_i - 1}) \\
    &+ z^{t_s+1} \prod_{i=1}^{s-1} (1 + z + \cdots + z^{t_i - 1}).
\end{align*}
Consequently we have cancellation between the middle and last terms, and we are left with
\begin{equation*}
    h(R/I_G ; z) = \prod_{i=1}^s (1 + z + \cdots + z^{t_i}) - z \prod_{i=1}^s (1 + z + \cdots + z^{t_i-1}). \qedhere
\end{equation*}
\end{proof}

Finally, we compute the Castelnuovo--Mumford regularity of $\mathbb{K}[G(t_1,\ldots,t_s)]$.
Recall that for any homogeneous ideal
$I$ of $R$, the {\it Castelnuovo-Mumford
regularity} of $R/I$ is
$${\rm reg}(R/I) = \max\{j-i ~|~
\beta_{i,j}(R/I) \neq 0\}$$
where $\beta_{i,j}(R/I)$ denotes the $(i,j)$-th 
graded Betti number that appears in the
minimal graded free resolution of $R/I$. 

\begin{proof}[Proof of Corollary~\ref{secondcor}]
Because the ideal $I_G$ with
$G = G(t_1,\ldots,t_s)$ is
geometrically vertex decomposable by Theorem~\ref{maintheorem},
the ideal $I_G$ is also Cohen--Macaulay by \cite[Corollary 4.5]{KR2021}.
For any Cohen--Macaulay ideal $I$,
it is well-known (see,
for example, \cite[Corollary B.28]{Vas1998}) that
${\rm reg}(R/I) = \deg h(R/I;z).$
The result now follows immediately
from Corollary \ref{maincor}.
\end{proof}

\subsection*{Acknowledgments}

Part of the work on this project was
carried out at the Fields Institute in Toronto, Canada
as part of the ``Thematic Program on Commutative Algebra
and Applications''.   All of the authors thank
Fields for providing an wonderful environment in which to work. Additionally, Van Tuyl
and Zotine thank the Fields Institute for 
funding to participate in this program.
Van Tuyl’s research is supported by NSERC Discovery Grant 2024-05299.

\bibliographystyle{plain}
\bibliography{references.bib}

\begin{thebibliography}{10}

\bibitem{ADS}
Ayah Almousa, Anton Dochtermann, and Ben Smith.
\newblock Root polytopes, tropical types, and toric edge ideals.
\newblock {\em Algebr. Comb.}, 8(1):59--99, 2025.

\bibitem{BHSD2023edgerings}
Kieran Bhaskara, Akihiro Higashitani, and Nayana Shibu~Deepthi.
\newblock The $h$-vectors of edge rings of odd cycle compositions.
\newblock {\em J. Algebra Appl.}, 2023.
\newblock paper no. 2550362, 2025.

\bibitem{BOKVT2017}
Jennifer Biermann, Augustine O'Keefe, and Adam Van~Tuyl.
\newblock Bounds on the regularity of toric ideals of graphs.
\newblock {\em Adv. in Appl. Math.}, 85:84--102, 2017.

\bibitem{BH}
Winfried Bruns and J\"{u}rgen Herzog.
\newblock {\em Cohen--{M}acaulay rings}, volume~39 of {\em Cambridge Studies in Advanced Mathematics}.
\newblock Cambridge University Press, Cambridge, 1993.

\bibitem{CDSRVT2023}
Mike Cummings, Sergio Da~Silva, Jenna Rajchgot, and Adam Van~Tuyl.
\newblock Geometric vertex decomposition and liaison for toric ideals of graphs.
\newblock {\em Algebr. Comb.}, 6(4):965--997, 2023.

\bibitem{DH2023}
Sergio Da~Silva and Megumi Harada.
\newblock Geometric {Vertex} {Decomposition}, {Gr{\"o}bner bases}, and {Frobenius} {Splittings} for {Regular} {Nilpotent} {Hessenberg} {Varieties}.
\newblock {\em Transformation Groups}, pages 1--36, 2023.

\bibitem{HHBOK2019}
Huy~T\`ai H\`a, Selvi~Kara Beyarslan, and Augustine O'Keefe.
\newblock Algebraic properties of toric rings of graphs.
\newblock {\em Comm. Algebra}, 47(1):1--16, 2019.

\bibitem{HHKOK2011}
Takayuki Hibi, Akihiro Higashitani, Kyouko Kimura, and Augustine~B. O'Keefe.
\newblock Depth of edge rings arising from finite graphs.
\newblock {\em Proc. Amer. Math. Soc.}, 139(11):3807--3813, 2011.

\bibitem{HD2023}
Akihiro Higashitani and Nayana~Shibu Deepthi.
\newblock The {$h$}-vectors of the edge rings of a special family of graphs.
\newblock {\em Comm. Algebra}, 51(12):5287--5296, 2023.

\bibitem{KR2021}
Patricia Klein and Jenna Rajchgot.
\newblock Geometric vertex decomposition and liaison.
\newblock {\em Forum Math. Sigma}, 9:Paper No. e70, 23, 2021.

\bibitem{KMY2009}
Allen Knutson, Ezra Miller, and Alexander Yong.
\newblock Gr\"obner geometry of vertex decompositions and of flagged tableaux.
\newblock {\em J. Reine Angew. Math.}, 630:1--31, 2009.

\bibitem{NRVT2024gvdideals}
Th{\'a}i~Th{\`a}nh Nguy{\~{\^e}}n, Jenna Rajchgot, and Adam Van~Tuyl.
\newblock Three invariants of geometrically vertex decomposable ideals.
\newblock {\em Pacific J. Math.}, 333(2):357--390, 2024.

\bibitem{Vas1998}
Wolmer~V. Vasconcelos.
\newblock {\em Computational methods in commutative algebra and algebraic geometry}, volume~2 of {\em Algorithms and Computation in Mathematics}.
\newblock Springer-Verlag, Berlin, 1998.
\newblock With chapters by David Eisenbud, Daniel R. Grayson, J\"urgen Herzog and Michael Stillman.

\bibitem{V}
Rafael~H. Villarreal.
\newblock {\em Monomial algebras}.
\newblock Monographs and Research Notes in Mathematics. CRC Press, Boca Raton, FL, second edition, 2015.

\end{thebibliography}

\end{document}